 \documentclass[reqno]{amsart}
\usepackage{amssymb}

\usepackage{graphicx,amssymb,amsfonts,epsfig,amsthm,a4,amsmath,url}
\usepackage[latin1]{inputenc}
\usepackage{pdfsync}

\newtheorem{thm}{Theorem}[section]
\newtheorem{cor}[thm]{Corollary}
\newtheorem{lem}[thm]{Lemma}
\newtheorem{clai}[thm]{Claim}
\newtheorem{prop}[thm]{Proposition}
\newtheorem{prob}{Problem}
\theoremstyle{definition}
\newtheorem{defn}[thm]{Definition}

\theoremstyle{remark}
\newtheorem{rem}[thm]{Remark}
\newtheorem{ex}[thm]{Example}

\numberwithin{equation}{section}

\newcommand {\BB}{\mathcal{B(\HH)}}
\newcommand{\Z}{\mathbb{Z}}

\newcommand{\F}{\mathbf{F}}

\newcommand{\N}{\mathbb{N}}
\newcommand{\R}{\mathbb{R}}
\newcommand{\C}{\mathbb{C}}

\newcommand {\PR}{\Pi(\HH)}
\newcommand{\VV}{\mathcal V}
\newcommand{\EV}{{\bf V}}

\newcommand {\CC}{\mathcal{C}}

\newcommand{\QC}{{\CC}}
\newcommand{\Pos}{\mathcal{P}^+(\HH)}

\newcommand{\conv}{\text{co}}
\newcommand{\diag}{\text{diag}}
\newcommand{\la}{\langle}
\newcommand{\ra}{\rangle}
\newcommand{\wst}{weak operator topology}
\newcommand{\T}{\mathcal{T}}

\newcommand{\corank}{\text{corank}}

\newcommand{\Id}{\text{Id}}
\newcommand{\spa}{\text{span}}
\newcommand{\HH}{\mathcal{H}}

\newcommand{\W}{\mathcal{W}}

\newcommand{\eps}{\varepsilon}

\newcommand{\trace}{\textnormal{tr}}

\begin{document}

\title{On the existence of Optimal Subspace Clustering Models}

\author[Akram Aldroubi]{Akram~Aldroubi} \address[Akram Aldroubi]{Department of Mathematics\\
Vanderbilt University\\ 1326 Stevenson Center\\ Nashville, TN 37240}
\email 
{akram.aldroubi@vanderbilt.edu}

\author[Romain Tessera]{Romain~Tessra} \address[Romain Tessera]{UMPA, ENS Lyon\\
Vanderbilt University\\ 146 Alle?e dÕItalie\\69364 Lyon Cedex, France}
\email 
{rtessera@umpa.ens-lyon.fr}

 \thanks{ The research of A.~Aldroubi is supported in part by NSF Grant DMS-0807464.}

\date{\today}



\maketitle
\begin {abstract}
Given a set of vectors $\F=\{f_1,\dots,f_m\}$   in a Hilbert space $\HH$, and given a family  $\CC$ of closed 
subspaces of $\HH$, the {\it subspace clustering problem} consists in finding a union of subspaces
in $\CC$ that best  approximates (models) the data $\F$. 
 This problem has applications and connections to many areas of mathematics,
computer science and engineering such as the Generalized Principle
Component Analysis (GPCA), learning theory, compressed sensing,
and sampling with finite rate of innovation. In this paper, 
we characterize  families of subspaces $\CC$ for which
such a best approximation exists. In finite dimensions the characterization is in terms of  the convex hull of  an augmented set $\CC^+$. In infinite dimensions however, the characterization is in terms of a new but related notion of contact hull. As an application, the existence of best approximations from  $\pi(G)$-invariant families $\CC$ of  unitary representations of abelian groups is derived.   
\end {abstract}
\section {Introduction}

\subsection {Motivation}

Let $\HH$ be a Hilbert space, $\F=\{f_1,\dots,f_m\}$  as set of vectors in $\HH$, $\CC$ a
family of closed subspaces of $\HH$, $\VV$ the set of all
sequences of elements in $\CC$ of length $l$ ( i.e.,
$\VV=\VV(l)=\big\{ \{ V_1,\dots,V_l\}:  V_i \in \CC, 1\le i\le l
\big\})$, and $\F$ a finite subset of $\HH.$ The
following problem has several applications in mathematics,
engineering, and computer science:
\begin{prob}[Non-Linear Least Squares Subspace Approximation]\label{NLproblem} ${ }$ \newline
\begin{enumerate}
\item Given a finite set $\F \subset \HH$, find the infimum of the expression
$$e(\F,\EV): =\sum\limits_{f \in \F} \min\limits_{1\leq j \leq l} d^2(f,V_j),$$
over $\EV=\{ V_1,\dots,V_l\} \in \VV$, where $d(x,y):=\|x-y\|_{\HH}$.
\item Find a sequence of subspaces $\EV^o=\{ V^o_1,\dots,V^o_l\} \in \VV$ (if it exists) such that
\begin {equation}\label{NLsolution}
e(\F,\EV^o) = \inf \{ e(\F,\EV): \EV \in \VV\}.
\end {equation}
\end{enumerate}
\end {prob}

\bigskip\noindent When $\HH=\R^d$, and $\CC$ is the family of all subspaces of dimension less than or equal to $r$, and $l=1$, Problem \ref {NLproblem} becomes the classical problem  of searching for a subspace $V$ of dimension $dim(V)\le r$ that best fit a (finite) set of data $\F\subset \R^d$. For this case, the best approximation $V$ exists and it can be obtained using the Singular Value Decomposition (SVD).

\bigskip

\noindent When $l>1$, we get a non-linear version of the single subspace approximation problem mentioned above. This non-linear version also  has many applications in mathematics and engineering. For example, Problem \ref {NLproblem} is related and has applications to the subspace segmentation problem known in computer vision (see e.g.,\cite {VMS05} and the references therein).  The subspace segmentation problem is used for face recognition, motion tracking in videos,  and the Generalized Principle Component Analysis GPCA \cite {VMS05}. Problem \ref {NLproblem} is also related and has application to segmentation  and spectral clustering of Hybrid Linear Models (see e.g., \cite {CL09} and the reference therein). Compressed sensing is another related area where finite signals in $\C^d$ are modeled by a union of subspaces $\mathcal{M}=\cup_{i\in I }V_i$, with $\dim V_i\le s$, where $s << d$ \cite {CRT06}.

\subsection {The Minimal Subspace Approximation Property}

 It has been shown that, given a family of closed subspaces $\CC$, the existence of a minimizing sequence of subspaces $\EV^o=\{V_1^o,\dots,V_l^o\}$ that solves Problem \ref{NLproblem} is equivalent to the existence of a solution to the same problem but for $l=1$ \cite{ACM08}:
\begin{thm}
Problem \ref{NLproblem} has a minimizing set of subspaces for any $l\ge 1$ if and only if it has a minimizing subspace for $l=1$.
\end{thm}
\noindent  This suggests the following definition:
\begin {defn}
\label {DEFMSAP}
A set of closed subspaces $\CC$ of a separable Hilbert space $\HH$ has the Minimum Subspace Approximation Property (MSAP) if for every finite subset $\F\subset \HH$ there exists an element $V\in \CC$ such that minimizes the expression $e(\F,V) = \sum_{f \in \F} d^2(f,V)$ over all $V\in \CC$.
More specifically, we will say that $\CC$ has $MSAP(k)$ for some $k\in \N$ if the previous property holds for all subsets $\F$ of cardinality $m\le k.$
\end{defn}
Using this terminology, Problem \ref {NLproblem} has a minimizing sequence of subspaces if and only if $\CC$ satisfies the MSAP.
 
 \begin{rem} We will see that, in general, $MSAP(k+1)$ is strictly stronger than $MSAP(k)$.
\end{rem}

There are some cases for which it is known that the MSAP is satisfied. For example, $\HH=\C^d$ and $\CC=\{V\subset \HH: \dim V\le s\}$, Eckhard-Young theorem \cite {EY36} implies that $\CC$ satisfies MSAP.  Another example is when  $\HH=L^2(\R^d)$ and $\CC=\overline{\spa \{\phi_1,\dots,\phi_r\}}$  is the set of all shift-invariant spaces  of length at most $r$. For this last example, a result in \cite {ACHM07} implies that $\CC$ satisfies the MSAP. However, a general approach for the existence of a minimizer had so far not been considered. Thus, the main goal of this paper is to provide necessary and sufficient conditions on the class $\CC$ of closed subspaces in $\HH$ so that $\CC$ satisfies the Minimum  Subspace Approximation Property.

\subsection{Organisation}

This paper is organized as follows: Section \ref {CMSAP} will be devoted to the characterization of MSAP in both finite and infinite dimensional Hilbert spaces.  We reformulate Problem \ref {NLproblem} in Section \ref {RP}. The characterization of MSAP is finite dimension is obtained in Section \ref{finitedimChar}. The infinite dimensional case is derived in Section \ref{infinitedimChar}.  A generalization of finitely generated shift-invariant spaces is discussed in Section \ref{GinvariantSection}, and a proof of  MSAP in these generalized spaces is derived. This gives a new and more conceptual proof of the fact that shift-invariant subspaces of length $\leq l$, for any $l$ satisfy MSAP which was proved in \cite{ACHM07}.


\bigskip

\section{Characterization of  MSAP }
\label {CMSAP}

Let $\CC$ be a set of closed subspaces of a separable Hilbert space $\HH$. Our main concern   is to find a topological characterization of the MSAP  of $\CC$. Recall that $\CC$ has MSAP if for every finite subset $\F\subset \HH$, there exists a subspace  $V^o\in \CC$ that minimizes the expression
\begin {equation}
\label {CF0}
e(\F,V)=\sum_{f\in \F}d^2(f,V),
\end {equation}
over all $V \in \CC.$

\subsection{Reformulation in terms of projectors and states}
\label {RP}
Let $\PR\subset \BB$ denotes the set of orthogonal projectors. The   subspaces in $\CC$ will be  identified with a subset of projectors. Specifically, a subspace $V\in \CC$ will be identified  with the orthogonal projector $Q=Q_V$ whose kernel is exactly $V$ (i.e., $Q=I-P_V$ where $P_V$ is the orthogonal projector on $V$). In this way,  we will identify  any set of closed subspaces $\CC$ of $\HH$  with a set of projectors $\{Q \in \PR:\ker (Q)\in \CC\}\subset \PR$. Using this identification,  we will denote this set of projectors by $\CC$ as well, and  express $e$ in \eqref {CF0} as
\begin {equation}
\label {CF}
e(\F,V)=\Phi_\F(Q_V)=\sum_{f\in \F}\la Q_V f,f\ra ,
\end{equation}
where $\Phi_F$ can now be viewed as a linear functional on the set $\BB$ of bounded linear operators on $\HH$.

  The choice of the topology to be considered on the set of bounded linear operators $\BB$ is of crucial importance for analyzing the optima of the functional $\Phi_\F$. Since we are looking for problems of existence of minimizers, compacity will be of great help.  
Hence, the weak operator topology will be the right one considered here. Namely, recall that operators of norm $\leq 1$ form a compact subset for the weak operator topology (this is no longer true with the strong operator or the norm topologies when the dimension is infinite). However, in infinite dimensional spaces, the set or projectors {\it is not}
closed for the weak operator topology. This will create most of the complication that we will have to face in infinite dimension.
In the sequel, a ``closed" subset of $\BB$ will always mean closed for the weak operator topology.

From its expression, $\Phi_\F$ is a continuous linear functional in the weak operator topology of  $\BB$.
Since the set of orthogonal projectors $\PR$ is a bounded set in the uniform topology of $\BB$, it is pre-compact in the \wst\  of $\BB$. Thus it is evident that if the set $\QC \subset \PR$ is  weakly closed in $\BB$, then $\Phi_\F$ attains its minimum (and maximum) for some $V^o\in \CC$, and we get
\begin {prop}
\label {Pclosedprop}
 If the set $\QC \subset \PR$ is  closed in the \wst\ of $\BB$, then $\CC$ satisfies MSAP.
\end {prop}
 Note that if we let $\alpha:=\sum_{f\in \F}\| f\|^2\ne 0$, then $\alpha^{-1}\Phi_\F$ is a {\em state} on $\BB$, i.e., it is  a linear functional on $\BB$ which
is non-negative on the subset $\Pos$ of non-negative
self-adjoint operators of $\BB$, and which equals $1$ on the
identity operator. We can therefore reformulate MSAP as follows: $\CC$ has MSAP if every state reaches its infimum on $\CC$.

\subsection{A sufficient geometric condition for  MSAP}

Clearly,   it is not necessary for $\CC$ to be  closed in order for $\CC$ to have the MSAP even if $\HH$ has finite dimension. For example, the theorem below gives a sufficient condition for MSAP  in terms of the set $\QC^+=\QC+\Pos$  where $\Pos$ is the set of non-negative self adjoint operators. It also gives a sufficient condition for MSAP  in terms of the convex hull  $\conv(\QC)$ of $\QC$, i.e., the smallest convex set containing $\QC$:

\begin{thm}
\label{topologicalPropMAP}
$$(\QC=\overline{\QC}) \Rightarrow (\QC^+=\overline{\QC^+}) \Rightarrow  (\conv(\QC^+)=\conv(\overline{\QC^+})) \Rightarrow (\textnormal{$\CC$ satisfies $MSAP$}),$$
and the implications are strict in general.
\end{thm}

\begin {proof}
Since $\overline {\QC}$ is compact and $\Pos$ is closed, $\overline {\QC}+\Pos$ is closed. Hence $\QC^+=\overline {\QC}+\Pos=\overline{\QC+ \Pos}=\overline {\QC^+}$, and the first implication follows. The second implication is obvious. For the last implication, let  $F$ be a finite set in $\HH$.
Recall that $\Phi_\F$ is continuous for the weak topology on $\BB$.
We make the following trivial observation:  if $Q \in \BB$ is a convex combination of elements $Q^1, Q^2, \ldots$, then there exists $i$ such that $\Phi_F(Q^i)\leq \Phi_F(Q).$
In particular, we have that $$\inf_{\conv(\overline{\QC})}\Phi_\F=\inf_{\overline{\QC}}\Phi_\F,$$ and since $\overline{\QC}$ is compact, this infimum is attained for some $R_0\in \overline{\QC}$.
By continuity of $\Phi_\F$,   we have $\Phi_\F(R_0)=\inf_{Q\in
\QC}\Phi_\F(Q)=\inf_{Q\in
\overline {\QC}}\Phi_\F(Q)$. Using the assumption $\conv(\QC^+)=\conv(\overline{\QC^+})$, and the easily checked equality  $\conv(\QC^+)=\conv(\QC)+\Pos$, we get that $\conv(\overline{\QC}) \subset \conv(\QC)+\Pos$. Thus we have  that $R_0=Q_0+N$ for some $Q_0 \in \conv(\QC)$ and $N \in \Pos$. It follows that
$$\Phi_\F(Q_0)\le \Phi_\F(Q_0) + \Phi_\F (N)=\Phi_\F(R_0)\le \Phi_\F(Q_0).$$ Hence
$\Phi_\F(Q_0)=\Phi_\F(R_0)=\inf_{Q\in \conv(\QC)}\Phi(Q)$.
But since $Q_0\in \conv(\QC)$, there exists $P_0\in \QC$ such that $\Phi_F(P_0)=\Phi_\F(Q_0)$. This proves that $\CC$ satisfies the MSAP.

\bigskip

 To see why the last implication cannot be reversed, take  $\CC$ to be the set of all finite dimensional subspaces (except the trivial vector space) of some infinite dimensional space $\HH$. It has MSAP since for all finite sets $F$ in $\HH$, one can find a finite dimensional subspace containing it. On the other hand the convex hull of $\QC^+$ does not contains $0$, while the weak closure of $\QC$ does.

\bigskip

To show that the second implication is not an equivalence, let $\HH=\ell^2$ and consider the sequence of vectors: $v_n=e_1+e_n$ which weakly converges to $e_1$, and
the sequence $w_n=e_2+ e_{n+1}$, which weakly converges to $e_2$.
For all $n\geq 3$, let  $P_n$ be the projector on the space spanned by $v_n$ and $w_n$. One checks that the sequence $\{P_n\}$  converges weakly to $Q=(P_{E_1}+P_{E_2})/2$, where $E_1=span \{e_1\}$, and $E_2=span \{e_2\}$. Moreover,  since $P_n=Q+ (P_{E_n}+P_{E_{n+1}})/2,$ where $E_n=\spa \{e_n\}$, and $E_2=\spa \{e_{n+1}\}$, we conclude that $Q< P_n$ for any $n$.
Now define $\QC=\{P_n,\; n\geq 3\}\cup \{P_{E_1},P_{E_2}\}.$
The closure of $\QC$ consists of $\QC\cup \{Q\}$. By the previous remark, $Q$ does not belong to $\QC^+$, so that $\QC^+$ is not closed. But on the other hand  $Q\in \conv(\QC)$, hence $\overline \QC \subset \conv (\QC)$ which implies that $\conv(\QC)=\conv(\overline{\QC})$.  It follows that  $\conv(\QC^+)=\conv(\overline{\QC^+})$.

\bigskip

Finally, to see why the first implication cannot be reverse, let $\HH=\R^3$ and consider the set $\CC=\CC_1\cup \CC_2$ which is the union  of the plane $\CC_1=\spa \{e_1,e_2\}$ and the set of lines $\CC_2= \cup_{v} \{\spa \{v\}: v= e_3+ce_2, \text{for some } c \in \R\}$. Then $\QC$ (identified with  a set of projectors as described earlier) is not closed  (since $Q_{\spa\{e_2\}} \notin \CC$) but $\CC$ satisfies the MSAP, since if the infimum is achieved by the missing line $\spa\{e_2\}$, it is also achieved by the plane $\CC_1$.
\end {proof}

\begin {rem}
For a perhaps more convincing example of why the last implication is not an equivalence, let $\HH=\ell^2$, $\{e_j\}$ its canonical basis, and let $\CC$ be the set of subspaces of co-dimension one except for the space $\spa\{e_1\}^{\perp}$. In other words,  $\QC$ consists of all projectors of rank one except the projector $Q_1$ on $\spa\{e_1\}$. Clearly for any finite subset $\F$ there exist
 (infinitely many) subspaces of $\CC$ on which $\Phi_\F$ vanishes, i.e., $\CC$ has MSAP.  However, $\conv(\QC^+)$ is not closed since $Q_1$ does not belong to it, while being in the closure of $\QC$.
 \end {rem}

\subsection{Characterization of MSAP in finite dimension}\label{finitedimChar}

 The fact that the first implication in Theorem \ref {topologicalPropMAP}  is not an equivalences was shown using an example in a finite dimensional space, while the other two examples showing the the other two implications are not equivalence were obtained in  an infinite dimensional space. It turns out that  for finite dimensional spaces the last two implications are in fact equivalences, and that for a Hilbert space $\HH$ of dimension $d<\infty$, MSAP is also equivalent to $MSAP(d-1)$ (see Definition \ref {DEFMSAP}). We have
 \begin {thm} Suppose $\HH$ has dimension $d$. Then the following are equivalent
 \begin{itemize}
\item[(i)] $\CC$ satisfies MSAP(d-1);
\item[(ii)] $\CC$ satisfies MSAP;
\item[(iii)]  $\conv(\QC^+)=\conv(\overline{\QC^+})$;
\item[(iv)]  $\QC^+$ is  closed.
 \end{itemize}
 \label {FinteDNS}
  \end {thm}
 \begin{proof}
Clearly (ii) implies (i), and (iv) implies (iii). By Theorem \ref{topologicalPropMAP}, (iii) implies (ii). So we are left to show that (i) implies (iv).
Since for finite dimensions,  the set of projectors of rank $k$ is closed, the zero projector belongs to $\QC$ if and only if it belongs to $\overline{\QC}.$ 
Note that the case where $\QC$ contains the zero projector is trivial, since this implies that $\QC^+$ is all of $\Pos$.  We shall therefore omit this case. Now let $R_0\in \overline{\QC}$ be a non-zero projector, and let
 $\F\subset \HH$ be a basis for
$\ker R_0$. Since $R_0$ is non-zero,  $\ker R_0$ has dimension $\leq d-1$  so that $\# (\F)\le d-1$. Then $\Phi_\F(R_0)= 0=\inf_{Q\in
\QC}\Phi_\F(Q)$. Since $\CC$ satisfies MSAP(d-1) there exists $Q_0\subset \QC$ such that $\Phi_\F(Q_0)=0$. Hence $\ker R_0\subset \ker Q_0$, so that $Q_0 \le R_0$, i.e., $R_0=Q_0 +N$ for some $N \in \Pos$.  Thus, $ \overline{\QC}\subset \QC^+=\QC+\Pos$. We get that  $ \overline{\QC} +\Pos\subset \QC+\Pos+\Pos=\QC^+$. Since $\Pos$ is closed and $ \overline{\QC} $ is compact, it follows that $ \overline{\QC} +\Pos $ is closed. Hence $\QC^+$ is closed.
 \end{proof}


 \bigskip
\begin {rem}
\noindent For spaces of dimension $d< \infty$, it is easy to show that MSAP(k-1) is strictly weaker than MSAP(k) for $1\le k \le d-1$. To see this, let $\mathcal C$ be the set of all subspaces of dimension $k$ except for the subspace generated by $k$ independent vectors $\{v_1,\dots,v_k\}.$ The $\mathcal C$ has MSAP(k-1), but it does not satisfy MSAP(k).
\end {rem}

\subsection{Characterization  of MSAP in infinite dimension}\label{infinitedimChar}

Let $X$ be a locally convex topological vector space. The example that we have in mind here is $\BB$ equipped with the weak-operator topology.
\begin{defn}
We define a closed half-space to be a set of the form $H_{\phi,a}=\{x\in X: \,  \phi(x)\geq a\}$, for some $a\in \R$ and $\phi$ an $\R$-linear functional on $X$.  The boundary of such half-space is the (affine) hyperplane of equation $\phi(x)=a.$
\end{defn}
For the reminder of this paper, we will use half-space to mean closed half space.
 Using this terminology, a direct application of Hahn-Banach theorem implies that the closed convex hull $\overline {\conv}(E)$ of a subset $E$ is the intersection of all half-spaces containing $E$.  Let us introduce a slightly different notion, called contact hull.
\begin{defn}
The contact hull $T(E)$ of a subset $E$ of $X$ is the intersection of all half spaces containing $E$ and whose boundary intersects $E$ non-trivially. Such a half-space is called contact half-space of $E.$ The set of contact half-spaces containing $E$ is denoted by $\T(E)$.
\end{defn}

\begin{rem} Note that if the set of all contact half-spaces of $E$ is empty, then $T(E)=X$. Clearly the contact hull is closed and the closed convex hull is contained in the contact hull. However, the converse is not true.
\end{rem}

Now, we are ready to prove a geometric characterization of MSAP. Here the locally convex space $X$ is $\BB$ equipped with the weak-operator topology.

\begin{thm}
Let $\QC$ be a set of projector in $\BB.$ Then $\QC$ has MSAP if and only if $$\T(\QC^+)= \T\big(\overline{\QC^+}\big).$$
\end{thm}
\begin{proof} First note that  a $\R$-linear  functional $\phi\in \BB^*$ reaches its minimum $a\in \R$ on some subset $E$  if and only if  $\{x, \phi(x)\geq a\}$ is a contact half-space of $E$.

\bigskip

 Let us show that  $\T(\QC^+)=\T\big (\overline{\QC^+}\big)$ implies MSAP.  Consider some $\R$-linear functional  $\Phi_{\F}$ for some finite set $\F$. By definition of the weak operator topology on $\BB$, $\Phi_{\F}$ is continuous for this topology. Hence the infimum $m$ of $\Phi_{\F}$ on $\QC^+$ coincides with its infimum on $\overline{\QC^+}=\overline{\QC}^+$. But the infimum $m$ on  $\overline{\QC}^+$ is achieved on $\overline{\QC}$, since  $\overline{\QC}$ is compact. Thus, $H_{\Phi_{\F} ,m} \in \T(\overline{\QC^+}) .$ But saying that  $\T(\QC^+)=\T(\overline{\QC}^+)$ exactly means that the contact half-spaces to $\QC^+$ and those of $\overline{\QC}^+$ coincide. This  means that $H_{\Phi_{\F} ,m} \in \T(\QC^+),$ and hence $\Phi_{\F}$ reaches its minimum on $\QC^+$, and therefore on $\QC$.

\bigskip

 To prove the converse, we will need the following well-known (but crucial) observation. In terms of states, it says that states which are continuous for the weak operator topology are convex combination of vector states, i.e.,  a convex combinations of functionals on $\BB$ of the form $\phi= \la \cdot f, f\ra$ where $f \in \HH$ and $\|f\|=1.$

\begin{clai}
For a real or a complex Hilbert space, the restriction  to $\Pos$ of an element $\phi\in \BB^*$ which takes non-negative values on $\Pos$ can be written as  $\phi(A)=\Phi_\F(A)$ for some finite set $\F$, and for all $A\in \Pos$.
\end{clai}

\begin{proof}
Since $\phi$ is a continuous functional on $\BB$ endowed with the weak operator topology, $\phi$ must be of the form $$ \phi(A)= \sum\limits_{i\in I} \langle Av_i,w_i  \rangle$$ for some vectors $\{v_i,w_i: i \in I\}$ in $\HH$, where $I$ is a finite set. Let $M:=\spa \{v_i,w_i:  i\in I\}$ and the operator $A_M:=\Pi_MA\Pi_M$ where $\Pi_M$ is the orthogonal projector on $M.$ Define the operator   $\phi_M(A_M)=\phi(A)$ for all $A \in \BB$.  Let $\{e_i\}$ be an orthonormal basis for $M$. Then $\phi_M(B)=\sum_{i,j}\lambda_{ij}  \langle Be_i,e_j  \rangle$ for all $B\in \mathcal B (M)$. This can be written as $\trace(\Sigma \Omega)$, where $\Sigma$ is the matrix with entries $\lambda_{ij}$ and $\Omega$ is the matrix whose entries are $\langle Be_j,e_i  \rangle$. Since $\phi$ is nonnegative on $\Pos$ it follows that for any nonnegative definite matrix $\Omega$
$$\trace(\Sigma \Omega)=\trace\big((\frac {\Sigma+\Sigma^*}{2})\Omega\big)=\trace \Lambda S^*\Omega S,$$
where $\Lambda$ is the matrix of eigen values of $(\frac {\Sigma+\Sigma^*}{2})$ and $S$ is the matrix of its eigen vectors. Therefore, $\phi(A)=\Phi_\F(A)$ for all $A\in \Pos$, where the vectors $\F=\{f_1,\dots,f_m\}$ are given by $f_i=\sigma^{1/2}\sum_jS_{ij}e_j$.
\end{proof}



\bigskip

 Now suppose that $\QC$ has MSAP. Let $H_{\phi,a}=\{x, \phi(x)\geq a\}$ be a contact half-space to $\overline{\QC}^+$. Note that this implies that $\phi\geq 0$ on $\Pos$. Therefore, on $\Pos$,  $\phi=\Phi_{\F}$ for some $\F$.
As $\QC$ has  MSAP by assumption, $\Phi_\F$  reaches its minimum on $\QC$, hence on $\QC^+$. Therefore, the above half-space  $H_{\phi,a}$ is a contact half space  to $\QC^+$. 
\end{proof}

\begin{cor}
If $\CC$ has the MSAP, then $T(\QC^+)=\overline {\conv}(\QC^+).$
\end{cor}
\begin{proof}
Clearly $\overline {\conv}(\QC^+)\subset T(\QC^+)$. Moreover, as a general fact, $\overline {\conv}(\QC^+)=\overline {\conv}(\overline{\QC^+})$.
And since $\CC$ has MSAP, the previous theorem implies that $T(\QC^+)=T(\overline{\QC^+})$. Hence, without loss of generality, we can assume that $\QC$ is compact.

\bigskip
\noindent Now suppose that $x\in T(\QC^+)\smallsetminus \overline{\conv}(\QC^+)$. By Hahn-Banach's theorem there exists a weak-continuous $\R$-functional $\phi\in \BB^*$ such that $\phi(x)< \phi(y)$ for all $y\in \overline{\conv}(\QC^+)$.  Let $\lambda:=\inf_{\overline{\conv}(\QC^+)}\phi$. We have that $\lambda=\inf_{\overline{\conv}(\QC^+)}\phi=\inf_{\conv(\QC^+)}\phi=\inf_{\QC^+}\phi=\inf_{\QC}\phi.$ Since $\CC$ is compact we have get that  $H_{\phi, \lambda}$ is a contact half-space to $\QC^+$.  In other words, $H_{\phi, \lambda}$ is a contact half-space containing $\QC^+$, but not $x$. This is in contradiction with the fact that $x\in T(\QC^+)$.
\end{proof}

\section{$\pi(G)$-invariant subspaces and MSAP}\label{GinvariantSection}

\subsection{Projectors of finite \corank}

Let $\HH$ denote a separable Hilbert space, and let, as before $\Pi=\PR$ be
the set of projectors on $\HH$, equipped with the topology of weak
convergence. Let $\mathcal P^+_1$ denote the set of positive
self-adjoint operators of norm $\leq 1$. Note that  $\Pi$ is
 closed for the strong topology but not for the weak one,
unless the dimension is finite, in which case $\Pi$ is compact.
Namely,  we have

\begin{prop}\label{weakProp}
Let $\Pi_k$ (resp. $\Pi_{-k}$) be the set of projectors of rank $k$ (resp. of corank $k$). If $\HH$ is infinite dimensional, then
\begin{enumerate}
\item The weak  closure of $\Pi_k$ is the set of positive self-adjoint operators of norm $\leq 1$ and of rank at most $k$.

\item The weak closure of $\Pi_{-k}$ is the set of positive self-adjoint operators $q$ of norm $\leq 1$ such that $qx= x$ on a subspace $V$ of codimension at most $k$ (i.e., $\dim V^\perp\le k$). In other words, $q\in \overline{\Pi_{-k}}$ if and only if  $q$ is a positive operators of norm $\|q\|\leq 1$ such that $q \geq p$ for some $p\in \Pi_{-k}.$

\end{enumerate}
\end{prop}

To prove Proposition \ref {weakProp} we need the following Lemma:
\begin{lem}
The operators of rank $\leq k$ form a weak closed subset of $\BB.$
\end{lem}
\begin {proof} [Proof of Lemma] This follows from the fact that an operator $x$ has a rank less than or equal to $ k$ if and only if all sub-determinants of size $k+1$ of
$x$ vanish (for instance if $x$ is expressed as a matrix in an orthonormal
basis), and the fact that the sub-determinant functions of any size $n$ are weakly continuous.\end{proof}

\begin{proof} [Proof of Proposition \ref {weakProp}] Note that the two  statements are  equivalent. Thus we only need to prove (i).

\bigskip
\noindent
Note that the set of positive operators, the set of operators with norm $\leq 1$, and by the lemma, operators of rank
$\leq k$, are weakly closed sets. So we only have to show that
every positive self-adjoint operator $x$ of norm $\leq 1$ and of
rank at most $k$ can be weakly approximated by projectors of rank
$\leq k$. To see that, let us express $x$ as a diagonal matrix in an
orthonormal basis $(e_i)_{i\in \N}$, such that the diagonal
coefficients $x_{i,i}=t^2_i$ lie in $[0,1]$, and equal $0$ if $i\geq
k+1$. For every unit vector $v$, denote by $p_v$ be the projector on
$\spa\{ v\}$. For each $i\in \N$, let $s_i=\sqrt{1-t_i^2}$. Clearly,
for every $i$ the sequence of unit vectors
$v_{n}(i)=|t_i|e_i+s_ie_{ni+n(k+1)}$ converges weakly to $|t_i|e_i$.
Hence the sequence of projectors $p_{v_n(i)}$ converges weakly to
$t^2_ip_{e_i}.$ Moreover $v_n(i)\perp v_n(j)$ for $i\ne j$. Therefore the sequence $x_n=\sum_{1\leq
i\leq k}p_{v_n(i)}$ is a sequence  of projectors of rank $k$ that converges weakly to $x=\sum_{1\leq
i\leq k}t^2_ip_{e_i}$.
\end{proof}

\noindent Note that (by definition) the space $\Pi_{-k}^+$ is the set of positive self-adjoint $q$ such that there exists a projector $p$ of corank at most $k$ such that  $q \geq p$. As an immediate consequence of (2) in the above proposition, we obtain
\begin{cor}\label{finitecorankThm}
We have that $\overline{\Pi_{-k}^+}=\Pi_{-k}^+$,
so, in particular, the set $\CC$ of subspaces of dimension at most $k$ has MSAP.
\end{cor}

\subsection{The $\pi(G)$-invariant setting}

 Let $\HH$ denotes  a Hilbert space (possibly of finite dimension), and let $G$ be a locally compact, $\sigma$-compact group, and let $\pi$ be a continuous unitary representation of $G$, i.e. a continuous homomorphism $G\to \BB$. 

\begin{defn}
Let $W$ be a closed $\pi(G)$-invariant subspace of $\HH$, i.e., $\pi(g)w\in W$ for all $g\in G $ and $w\in W$.
We define the $\pi(G)$-dimension of $W$ to be the minimal dimension of a subspace $V$ such that $$W=\overline{\spa\{\pi(g)v;\; g\in G, v\in V\}}.$$
\end{defn}

\begin{ex} 
The prototypical example is $\HH=L^2(\R^d)$, $G=\Z^d$, and $\pi(k)$ being the translation operator by $k$. Given $\phi_1, \ldots, \phi_r\in L^2(\R)$, the space $W=\overline {\spa\{\phi_1(\cdot-k),\cdots,\phi_r(\cdot-k)\}}$ is a $\pi(G)$-invariant subspace of $L^2.$ Moreover, by construction, $W=\overline{\spa\{\pi(k)f;\; k\in \Z^d, f\in V\}},$ where $V=\spa\{\phi_1,\cdots,\phi_r\}$. Hence the $\pi(\Z^d)$-dimension of $W$ is $\leq r$. Such a space is often called a shift-invariant space of length $r$.
\end{ex}

Recall that an operator $A\in \BB$ is called $\pi(G)$-invariant if it commutes with the representation $\pi$, i.e. if
$\pi(g)A(w)=A(\pi(g)(w))$ for all $g\in G$ and $w\in \HH$.
Note that if $p$ is a projector, and if the representation is unitary, then $p$ is $\pi(G)$-invariant if and only if its range (resp. its kernel) is a $\pi(G)$-invariant subspace.

 In the ``$\pi(G)$-invariant setting", we can extend the notion of  corank in the obvious way, namely by saying that a $\pi(G)$-invariant operator has $\pi(G)$-corank $k$ if  its kernel has $\pi(G)$-dimension $k$.
We will denote the set of projectors of corank $\leq k$ by $\Pi_{-k}^G(\HH)$. 

\begin{rem}
Note that if the representation $\pi$ is trivial (i.e. if $\pi(g)=\Id$ for all $g\in G$), then the $\pi(G)$-dimension and $\pi(G)$-corank coincide with the usual dimension and corank.
\end{rem}

  Let us end this section by stating an analogue of Theorem \ref{topologicalPropMAP} in the $\pi(G)$-invariant setting. The proof is the same, using the fact that $\pi(G)$-invariant operators form a weak-closed subalgebra of $\BB$ (i.e. a Von Neumann subalgebra of $\BB$). Let us define $(\mathcal P^+)^G(\HH)$ to be the set of all $\pi(G)$-invariant, positive operators in $\BB.$ 

\begin{prop}
Suppose that $\CC$ is a set of $\pi(G)$-invariant subspaces of $\HH$.
If $\QC+(\mathcal P^+)^G(\HH)$ is closed in the \wst, then $\CC$ satisfies $MSAP$.
\end{prop}

\subsection{MSAP for invariant subspaces}

 The problem that we address now is whether we can generalize Corollary \ref{finitecorankThm} to the ``$\pi(G)$-invariant'' context. Precisely: does the set of all closed $\pi(G)$-invariant subspaces with $\pi(G)$-dimension at most $k$ has MSAP for all $k\in \N$?  
For simplicity, we will restrict our attention to the case of abelian groups. 
 
 As  mentioned earlier, an interesting case is the shift-invariant  spaces case in which  $G=\Z^d$ acts by translations on $\HH=L^2(\R^d)$. For this case the $\pi(G)$-invariant subspaces are usually called shift-invariant subspaces. Moreover, the $\pi(G)$-dimension is called {\it length} in this particular example. We now prove that the fact  that shift-invariant subspaces of length $\leq l$, for any $l$ satisfy $MSAP$  \cite{ACHM07} actually holds for any unitary representation of a locally compact abelian group $G$ (note that even for the group $\Z$, this yields a lot of new examples).

\begin{thm}\label{GinvThm}
Let $\pi$ be a unitary representation of a locally compact
$\sigma$-compact abelian group $G$.  Then, $$\overline{\Pi_{-k}^G}\subset \Pi_{-k}^G+(\mathcal P ^+)^G(\HH).$$ In particular, the set of closed $\pi(G)$-invariant subspaces of $\HH$ of $\pi(G)$-dimension at most $k$ has MSAP.
\end{thm}

Our main tool for proving this theorem, will be the generalized spectral theorem for locally compact $\sigma$-compact abelian groups (see for instance \cite{Dix}). 
Let $\hat{G}$ be the dual of $G$, i.e. the set of complex characters  ${\chi}$ on $G$. Recall that a character is a continuous homomorphism from $G$ to the group of complex numbers of modulus $1$. As this group identifies as $U(1)$, i.e. the group of unitary operators in dimension $1$, one sees that  $\hat{G}$ can be interpreted as the set of one-dimensional representations of $G$.
 Note that as $U(1)$ is a compact group, $\hat{G}$ comes naturally equipped with the structure of a locally compact, $\sigma$-compact group. To keep in mind some concrete examples, let us recall that the dual of $\Z$ is $U(1)$, that the dual of $U(1)$ is $\Z$ and that $\hat{\R}=\R$.

\subsection{Proof in finite dimension}

Although the proof in finite dimension is not different in spirit from the general case, we decided to add it for pedagogical reasons, since this allows us to avoid all the problems concerning measurability.

Let us start by recalling the content of the spectral theorem in finite dimension. Assume  $dim(\HH)=d$. Then there exists an orthonormal basis $(e_1,\ldots, e_d)$ where all operators $\pi(g)$ are diagonal in this basis. In other words there exists an isometry $T:\HH\to \ell^2(\{1,\ldots , d \})$ such that $T\pi(g)T^{-1}=\diag(\chi_i(g))$, where for each $1\leq i\leq d$, $\chi_i$ is a character of $G$.
Although this description of $\pi$ is quite nice, it could be improved to take into account the fact that various $\chi_i$'s might be equal. We now give another description.

\bigskip

\noindent{\bf The spectral theorem in finite dimension.}
Now, for each $\chi\in \hat{G}$, let us define the multiplicity $m_{\chi}$ of $\chi$ in $\pi$ to be the number of distinct values of $i$ such that $\chi=\chi_i$.
Let $\mu$ be the measure on $\hat{G}$ which equals a dirac on $\chi$ if  $m_{\chi}>0$, and zero otherwise. For each $\chi$, associate the space $\HH_{\chi}=\C^{m_{\chi}}$ (note that for all except finitely many this space is reduced to zero).
Now, define the space $L^2(\hat{G},(\HH_{\chi}),\mu)$ to be the direct sum $\bigoplus_{\chi} \HH_{\chi}$.  An element of this space is denoted $v=(v(\chi))_{\chi}$ such that $v(\chi)\in \HH_{\chi}$ for each $\chi\in \hat{G}.$ We equip  $L^2(\hat{G},(\HH_{\chi}),\mu)$ with the following scalar product:  for $v,w\in L^2(\hat{G},(\HH_{\chi}),\mu)$, we have
$$\langle v,w\rangle:=\int_{\hat{G}}\langle v(\chi)w(\chi)\rangle_{\HH_{\chi}} d\mu(\chi)=\sum_{\chi}\langle v(\chi),w(\chi)\rangle_{\HH_{\chi}}.$$
Let us now collect a few  observations. 

\begin{prop}

\

\begin{enumerate}
\item There is an isometry $S:\HH\to L^2(\hat{G},(\HH_{\chi}),\mu)$ such that  for all $g\in G$,  and all $v\in L^2(\hat{G},(\HH_{\chi}),\mu)$,
$$S\pi(g)S^{-1}(v)(\chi)=\chi(g)v(\chi).$$ To simplify the notation,  let us identify $\HH$ with $L^2(\hat{G},(\HH_{\chi}),\mu)$ and $\pi$ with $S\pi S^{-1}$, so that we now simply have
$$\pi(g)(v)(\chi)=\chi(g)v(\chi)$$
\item $A\in \BB$ is $\pi(G)$-invariant if and only if $A$ decomposes as $\bigoplus_{\chi\in \hat{G}}A_{\chi}$, with $A_{\chi}\in B(\HH_{\chi})$.
\item Suppose that $V$ is a subspace of $\HH$, and let $W$ be the $\pi(G)$-invariant subspace generated by $V$, i.e.  $W=\spa\{\pi(g)v; \; g\in G, v\in V\}.$ Then $W=\bigoplus_{\chi}V_{\chi}$ where each $V_{\chi}$ is the  projection of $V$  onto $\HH_{\chi}$.
\end{enumerate}
\end{prop}
\begin{proof} The proofs of (1) and (2) are elementary. Let us only indicate how to prove (3). Note that since $\spa\{\pi(g)v; \; g\in G, v\in V\}$ and that $\bigoplus_{\chi}V_{\chi}$ is $\pi(G)$-invariant we have that $W\subset \bigoplus_{\chi}V_{\chi}$. Moreover, since $V\subset W$, $V_\chi\subset W$, hence $\bigoplus_{\chi}V_{\chi}\subset W$. Thus, $W=\bigoplus_{\chi}V_{\chi}$. 
\end{proof}

\bigskip

Let us go back to the purpose of this section, namely the proof of Theorem \ref{GinvThm} when $\HH$ is finite-dimensional.
By the above proposition,  elements in $\Pi_{-k}^G$ are exactly elements of the form $p=\bigoplus_{\chi} p_{\chi}$, where each $p_{\chi}$ is  a projector on $\HH_{\chi}$ of corank at most $k$, i.e. an element of $\Pi_{-k}(\HH_{\chi})$.  Since $\HH$ is finite, there are only finitely $\chi$ for which $p_\chi$ is the nonzero projector. This latest fact together with the fact that for each $\chi$,  the set $\Pi_{-k}(\HH_{\chi})$ is closed, implies that $\Pi_{-k}^G$ is closed, and  Theorem \ref{GinvThm} follows when $\HH$ is finite-dimensional.

\subsection{Proof of Theorem \ref{GinvThm}}

As previously mentioned, the proof of  Theorem \ref{GinvThm} relies on the generalized spectral theorem for abelian groups that we now recall. The first two assertions of the following theorem follow from the classical theory (see for instance [Dix]).

\begin{thm}[Generalized spectral theorem for abelian groups]
Let $G$ be a locally compact $\sigma$-compact abelian group, and let $\pi$ be a unitary representation of $G$.
There
exists a Borel measure $\mu$ on $\hat{G}$, together with an isometry
from $\HH$ to $L^2(\hat{G},(\HH_{\chi}),\mu)$, where $(\HH_{\chi})_{\chi\in \hat{G}}$ is a measurable
family of  subspaces of a separable Hilbert space $\HH'$, such that the representation $\pi$
decomposes as a direct integral of representations
$(\HH_{\chi},\pi_{\chi})$, where $\pi_{\chi}$ is the multiplication by the character $\chi$. This is usually denoted by
$$\pi=\int^{\oplus}_{\hat{G}}\pi_{\chi}d\mu(\chi). $$
Moreover, the following statement are true.
\begin{itemize}
\item[(i)]  $A\in \BB$ is $\pi(G)$-invariant if and only if $A$ decomposes as $\int^{\oplus}_{\chi\in \hat{G}}A_{\chi}d\mu(\chi)$, with $A_{\chi}\in B(\HH_{\chi})$ (defined a.e.).
\item[(ii)]  Suppose that $V$ is a closed subspace of $\HH$, and let $W$ be the closed $\pi(G)$-invariant subspace generated by $V$, i.e.  $W=\overline{\spa\{\pi(g)v; \; g\in G, v\in V\}}.$ Then $W=\int^{\oplus}_{\chi}V_{\chi}d\mu(\chi)$ where each $V_{\chi}$ is the  projection of $V$  onto $\HH_{\chi}$ (defined a.e.).
\item[(iii)]  For every $\chi\in \hat{G}$, let $\W_{\chi}$ be a subset of operators of $B(\HH_{\chi})$ of norm $\leq 1$. Suppose moreover that the family $(\W_{\chi})_{\chi\in \hat{G}}$ is measurable. Consider $\W$ the subset of all ($\pi(G)$-invariant) operators of the form $\int^{\oplus}_{\hat{G}}A_{\chi}d\mu(\chi)$ such that $A_{\chi} \in \W_{\chi}$, which we denote by
$$\W:=\int^{\oplus}_{\hat{G}}\W_{\chi}d\mu(\chi).$$
Then, we have
$$\overline{\W}=\int^{\oplus}_{\hat{G}}\overline{\W_{\chi}}d\mu(\chi),$$
for the weak topology.
\end{itemize}
\end{thm}\label{SpectralThm}

\noindent {\it Proof of (iii)}. Assume $B\in \overline{\W}$, then there exists $A^n\in \W$ such that $A^n$ converges weakly to $B$. Since $B$ is also $\pi(G)$-invariant, $B=\int^{\oplus}_{\chi\in \hat{G}}B_{\chi}d\mu(\chi).$ Thus we have $\int^{\oplus}_{\chi\in \hat{G}}\big\langle (A^n_{\chi}-B_{\chi})v(\chi),w(\chi) \big\rangle_{\HH_{\chi}}d\mu(\chi) \to 0,$ as $n\to \infty.$ Hence $\|(A^n_{\chi}-B_{\chi})v(\chi) \|_{\HH_{\chi}}$ converges in measure to zero, from which we deduce that there exists a subsequence $(A_{\chi}^{n_j}-B_{\chi})v(\chi)$ that converges strongly to  $0$ a.e.. Hence $A_{\chi}^{n_j}$ converges weakly to $B_{\chi}$, a.e.. Thus, $B_{\chi}\in \overline{\W_\chi}$.  Conversely, assume that a sequence $A^n_{\chi} \in \W_{\chi}$ converges weakly to $B_{\chi}.$ Then the Lebesgue dominated convergence theorem implies that $A^n=\int^{\oplus}_{\chi\in \hat{G}}A_{\chi}d\mu(\chi),$ converges weakly to $B=\int^{\oplus}_{\chi\in \hat{G}}B_{\chi}d\mu(\chi).$     $\square$

\bigskip

The main step toward the proof of Theorem \ref{GinvThm} is the following description of $\pi(G)$-invariant projectors.

\begin{lem}\label{invariantprojClaim}
With the notation of Theorem \ref{GinvThm}, we have that for every $\pi(G)$-invariant projector $p$, there exists a measurable field $(p_{\chi})$ of projectors such that
$$p=\int^{\oplus}_{\hat{G}}p_{\chi}d\mu(\chi).$$
Moreover, the $\pi(G)$-corank of $p$ is at most $k$ (i.e., kernel of $p$ is $\pi(G)$-dimension $\leq k$) if and only if for a.e. $\chi\in \hat{G}$,
$corank (p_{\chi})\leq k.$
\end{lem}
\begin{proof}
By (ii) that if $p$ has $\pi(G)$-corank at most $k$, then for a.e. ${\chi}\in \hat{G}$,
$p_{\chi}$ has corank at most $k$.
The converse is also true but more subtle.

\begin{clai}
Let $\{p_x\}_{x\in X}$ be a measurable family of projectors of finite $\corank$ in some separable Hilbert space $\HH'$, where $X$ is some Borel measure space. For $i=1,2, \ldots,$  there exists a measurable family of vectors $\{e_i^x\}$ such that for a.e. $x\in X$, $(e_1^x, e_2^x, \ldots, e_{r_x}^x)$ is a basis of $\ker p_x$, and $e_i^x=0$ for $i>r_x$.
\end{clai}
\begin{proof} First, note that $r_x=\corank(p_x)$ is a measurable map $X\to \N$, and then define a measurable partition of $X=X_0\cup X_1\ldots$ such that $r_x$ equals $i$ on $X_i$ for each $i\in \N$.
So we can suppose that $r_x$ is constant, equal to some integer $k$.

Now, let us define a metric on $X$, by pulling back the norm on $\Pi=\Pi(\HH')$, i.e.
$d(x,y):= \|p_x-p_y\|.$
Clearly, $d$ is measurable. Since $\Pi_{-k}$ is a separable metric space, we can cover it by a countable family of balls of arbitrarily  small radius. Hence, one sees by an easy argument that it is enough to define $\{e_i^x\}$ on a small ball, say $B(x,\eps)$ for some $x\in X$, and $\eps>0$.
As $p_x$ is defined almost everywhere, either $B(x,\eps)$ has measure $0$, in which case, there is nothing to do, or there exists $y\in B(x,\eps)$ such that $p_y$ is well defined, and $r_y=k$. Take an orthonormal basis $(e_1, e_2, \ldots, e_k)$ of $\ker p_y$, and for all $z\in B(x,\eps)$, define
$e_i^z=(1-p_z)(e_i)$, for $i=1,\ldots, k.$ If $\eps$ is small enough, it is easy to see that  the $(e_i^z)_i$ are still linearly independent, and therefore form a basis of $\ker p_z$. \end{proof}

\bigskip

 Note that in the above claim, we can even suppose that the $(e_i^x)_{1\leq i\leq r_x}$ form an orthogonal basis of $\ker p_x$ for a.e. $x\in \hat{G}$, and that $\|e_i^x\|\leq 1$ for a.e.  $x\in \ker p_x$ and all $i$.
Now, if $p=\int^\oplus_{\hat{G}}p_{\chi}d\mu({\chi})$, with $\corank (p_{\chi})\leq k$, by the claim, we can find a measurable family of vectors $\{(e_1^{\chi}, e_2^{\chi}, \ldots,e_k^{\chi})\}_{\chi}$ such that  for a.e.  ${\chi}$, $(e_1^{\chi},\ldots, e_k^{\chi})$ generates $\ker p_{\chi}$, and are mutually orthogonal (some might be equal to $0$), and that $\|e_i^{\chi}\|\leq 1$ for a.e. $\chi$.  Now, let $f$ be a positive function in  $L^2(\hat{G},\mu)$, and define for a.e. $\chi$, $\xi_i(\chi)=f({\chi})e_i^{\chi}$. Note that $\xi_i\in L^2(\hat{G},(\HH_{\chi}))$ for all $i$.  We need to prove that $\ker p$ is generated, as a closed $\pi(G)$-invariant  subspace by $\xi_1, \ldots, \xi_k$. But since $p=\int^\oplus_{\hat{G}}p_{\chi}d\mu({\chi})$, if $W=\ker p$ then $W=\int^{\oplus}_{\chi}V_{\chi}d\mu(\chi),$ where $V_\chi=\ker p_\chi$. But by construction,  $V_\chi=\spa\{\xi_i(\chi):\; i=1,\cdots,k\}$. Thus, $\ker p$ has $\pi(G)$-codimension $k$. 
This finishes the proof of the lemma.
\end{proof}

\bigskip

Now we are in position to finish the proof of Theorem \ref{GinvThm}. By Theorem \ref{SpectralThm} (i), every $\pi(G)$-invariant projector
decomposes as $p=\int_{\hat{G}}p_{\chi}d\mu({\chi})$ where $(p_{\chi})_{{\chi}\in \hat{G}}$
is a Borel measurable family of projectors in $L^2(\hat{G},(\HH_{\chi}))$.
The proof of the theorem follows immediately from (iii), Corollary \ref{finitecorankThm} and the previous  lemma. 
$\square$

\subsection{Concrete examples and applications}
In this section we provide several examples and applications to illustrate the results.
\begin {ex} As a first application of Theorem \ref {topologicalPropMAP} or Corolloary \ref {finitecorankThm} we get the well known C.~Eckart and G.~Young Theorem \cite {EY36}. Specifically, if $\HH$ is finite dimensional and if and $\CC$ is the family of subspaces of dimension less than or equal to $r$, then   $\CC$ satisfies MSAP since $\CC$ is closed. As a consequence, Problem \ref {NLproblem} has a solution.
\end{ex}
\bigskip

\begin {ex}A more interesting example is a generalization of the C.~Eckart and G.~Young Theorem \cite {EY36} when $\HH$ is infinite dimensional and   $\CC$ is the family of subspaces of dimension less than or equal to $r$.  For this case, a direct application of Corollary \ref {finitecorankThm} implies that $\CC$ satisfied MSAP.
\end {ex}
\bigskip

\begin {ex} Another interesting example is  when $G=\Z$,  $\HH=L^2(\R)$, $\pi(l)f:=T_l f$ where $T_l$ is the translation operator by the translation $l$, and $\CC_k$ is the set of all subspaces that are $\pi(\Z)$-invariant and  and have $\pi(\Z)$-dimension $n$. The set $\CC_k$  is precisely the so called shift invariant subspaces of $L^2(\R)$ of length at most $k$, that acts by translations on $\HH=L^2(\R^d)$. As a direct consequence of Theorem \ref {GinvThm}, the set $\CC_k$ satisfies MASP. This gives a new proof for the fact that $\CC_k$ satisfies MSAP \cite {ACHM07}.
\end {ex}

\begin {ex} A discrete version of the shift invariant space result above can be obtained by letting $G=m\Z$ where $m$ is a positive integer,  $\HH=\ell^2(\Z)$, $\pi(ml)f:=S_{ml} x$ where $S_{ml}$ is the shift operator by the index $ml$,i.e., if $y=S_{ml}x$, then $y_n=x_{n-ml}$ for all $n \in \Z$.  If we let $\CC_k$ to be  the set of all subspaces that are $\pi(m\Z)$-invariant and  have $\pi(m\Z)$-dimension $k$. Then the set $\CC_k$ satisfies MSAP.
\end {ex}

\begin {ex}
An interesting and useful case is in the finite dimensional case $\HH=\ell^2(\{1,\cdots,d\})$, where $d=ml$ for some positive integers $m$ and $l$. Let $G=\Z/m\Z$ be the cyclic group of integers modulo $m$. By letting $\pi(j)x=S_{jl}x$ be the shift operator modulo $d$ we get that the set  $\CC_k$  of all subspaces that are $\pi(\Z/m\Z)$-invariant and  have $\pi(m\Z)$-dimension $k$ satisfy MASP.
\end{ex}

The last three examples are related in a way similar to the way the Fourier Integral, the Fourier series and the discrete FFT are related, and hence this relation can be used for  numerical approximations and implementations.

\bigskip
\footnotesize

\end{document}